\documentclass{article}

\usepackage[british]{babel}
\usepackage{amsmath}
\usepackage{amssymb}
\usepackage{amsthm}
\usepackage{verbatim}

\numberwithin{equation}{section}

\newcommand\norm[1]{\left\lVert#1\right\rVert}

\usepackage[toc,page]{appendix}

\usepackage{hyperref}
\usepackage[framemethod=tikz]{mdframed}

\usepackage{amsthm}

\newtheorem{defi}{Definition}[section]
\newtheorem{lem}[defi]{Lemma}

\newtheorem{theorem}[defi]{Theorem}
\newtheorem{propo}[defi]{Proposition}

\newtheorem*{theorem*}{Theorem}

\begin{document}

\title{Random walks in doubly random scenery}
\author{\L ukasz Treszczotko\footnote{Institute of Mathematics, University of Warsaw, Banacha 2 02-097 Warsaw}\\
\href{mailto:lukasz.treszczotko@gmail.com}{lukasz.treszczotko@gmail.com} }

\maketitle

\begin{abstract}
We provide a random walk in random scenery representation of a new class of stable self-similar processes with stationary increments introduced recently by Jung, Owada and Samorodnitsky. In the functional limit theorem they provided only a single instance of this class arose as a limit. We construct a model in which a significant portion of processes in this new class is obtained as a limit. 
\end{abstract}

\\

\textup{2000} \textit{Mathematics Subject Classification}: \textup{Primary: 60G18} \textup{Secondary: 60F17}

\maketitle

\section{Introduction}

\subsection{Random walks in random scenery}

Our model is based in the framework of \emph{random walks in random scenery models}. They were first considered in~\cite{KS}, where a number of limit theorems regarding the scaling limits of these models were proved. The more specific context in which we will be working was presented in~\cite{DG}. The model considered therein can be briefly sketched as follows. Assume that there is a \emph{user} moving randomly on the \emph{network} (in this paper the network is just $\mathbb{Z}$) which earns random rewards (governed by the random scenery) associated to the points in the network that they visit. The quantity of interest is then the total amount of rewards collected. To be more precise, assume that that the movement of the user is a random walk on $\mathbb{Z}$ which after suitable scaling converges to the $\beta$-stable L\'{e}vy process with $\beta \in (1,2]$. Furthermore, let the random scenery be given by i.i.d. random variables $(\xi_j)_{j\in \mathbb{Z}}$ which belong to the normal domain of attraction of a symmetric strictly stable distribution with index of stability $\alpha \in (0,2]$. Then the \emph{random walk in random scenery} is given by
\begin{equation}
Z_n=\sum_{k=1}^n\xi_{S_k},
\end{equation}
where $S_k=\sum_{j=1}^k X_k$ is the random walk determining the movement of the user. If we consider a large number of independent \emph{random walkers} moving in independent random sceneries, then the scaling limit in the corresponding functional limit theorem (see Theorem 1.2 in~\cite{DG}) leads to the process which has the integral representation given by
\begin{equation}\label{G1}
X=\left(\int_{\mathbb{R}\times \Omega'}L_t(x,\omega')M_\alpha(dx,d\omega')\right)_{t\geq 0},
\end{equation}
where $(L_t(x,\omega'))_{t\geq 0,x \in \mathbb{R}}$ is a jointly continuous version of the local time of the symmetric $\beta$-stable L\'{e}vy motion (defined on some probability space $(\Omega',\mathcal{F}',\mathbb{P}')$) and $M_\alpha$ is a symmetric $\alpha$-stable random measure on $\mathbb{R}\times \Omega'$ with control measure $\lambda_1 \otimes \mathbb{P}'$, which is itself defined on some other probability space $(\Omega,\mathcal{F},\mathbb{P})$. 
The process~\eqref{G1} was also obtained in~\cite{SAM1} where it arose as a limit of partial sums of a stationary and infinitely divisible process.

\subsection{The limit process}
Very recently Jung, Owada and Samorodnitsky in their paper~\cite{SAM2}, which was an extension of the model considered in~\cite{SAM1}, introduced a new class of self-similar stable processes whose members have an integral representation given by
\begin{equation}\label{G2}
Y_{\alpha,\widetilde{\beta},\gamma}(t):=\int_{\Omega'\times [0,\infty)}S_\gamma(M_{\widetilde{\beta}}((t-x)_+,\omega'),\omega')dZ_{\alpha,\widetilde{\beta}}(\omega',x),\quad t \geq 0,
\end{equation}
where 
\begin{equation*}
0<\alpha<\gamma\leq 2, 0\leq \widetilde{\beta}< 1,
\end{equation*}
$(S_\gamma(t,\omega'))_{t\geq 0}$ is a symmetric $\gamma$-stable L\'{e}vy motion and $(M_{\widetilde{\beta}}(t,\omega'))_{t\geq 0}$ is an independent $\widetilde{\beta}$-Mittag-Leffler process (see section 3 in~\cite{SAM1} for more on the latter). Both of these processes are defined on a probability space $(\Omega',\mathcal{F}',\mathbb{P}')$. Finally $Z_{\alpha,\beta}$ is a $S\alpha S$ random measure on $\Omega' \times [0,\infty)$ with control measure $\mathbb{P}'\otimes \nu_{\widetilde{\beta}}$, where $\nu_{\widetilde{\beta}}(dx)=(1-\widetilde{\beta})x^{-\widetilde{\beta}}\mathbf{1}_{x\geq 0}dx$. By Proposition 3.2 in~\cite{SAM2} the process $Y_{\alpha,\widetilde{\beta},\gamma}$ is $H$-sssi with Hurst coefficient $H=\widetilde{\beta}/\gamma+(1-\widetilde{\beta})/\alpha$. Here we use $\widetilde{\beta}$ instead of $\beta$ so as not to confuse it with the notation we have adopted for this paper. Similarly as in the proof of (3.10) in~\cite{SAM1} we can show that for $\widetilde{\beta}\in (0,\frac{1}{2})$
\begin{equation}\label{identi}
(Y_{\alpha,\widetilde{\beta},\gamma}(t))_{t\geq 0}\overset{d}{=}c_{\widetilde{\beta}}\Bigg(\int_{\Omega'\times \mathbb{R}}S_\gamma(L_t(x,\omega'),\omega')dZ_\alpha(\omega',x)\Bigg)_{t\geq 0},
\end{equation}
where $c_{\widetilde{\beta}}$ is a constant depending only on $\widetilde{\beta}$, $(L_t(x))_{t\geq 0}$ is the local time of a symmetric $\beta$-stable L\'{e}vy motion defined independent of the process $S_\gamma$ (both defined on $(\Omega',\mathcal{F}',\mathbb{P}')$), $\beta=(1-\widetilde{\beta})^{-1}$ and $Z_\alpha$ is a symmetric $\alpha$-stable random meaure on $(\Omega',\mathbb{R})$ with control measure $\mathbb{P}'\otimes \lambda_1$.\\

The limit process obtained in~\cite{SAM2} corresponds to $\gamma=2$ in~\eqref{G2} it is our purpose to provide a model in which the scaling limit is given by processes of the form~\eqref{identi} for any allowable choice of parameters $\alpha,\beta$ ans $\gamma$.

\section{Description of the model and the result}\label{DR}

Imagine that each $x \in \mathbb{Z}$ is associated with a reward (or punishment) given by $\xi_x$ which takes integer values. Now imagine a \emph{random walker} moving on $\mathbb{Z}$ independently of the rewards and starting at $0$. Before the movement the walker generates a strategy $Y_1,Y_2,\ldots$ of i.i.d. random variables which are independent of the $\xi_x$'s and his movement. Now, any time the walker visits a point $x$ he gets a reward (or receives punishment) given by $Y_k\times \xi_x$, where $k$ is number of times that the walker has already stayed at $x$ (including this time). Thus the amount by which a potential reward is being multiplied depends only on the number of the visits. The total reward/punishment at time $n$ in this scheme is given by
\begin{equation}
\sum_{x\in \mathbb{Z}}\Big(\sum_{k=1}^{N_n(x)}Y_k\Big)\xi_x,
\end{equation}
where 
\begin{equation}
N_n(x):=\sum_{k=1}^n\mathbf{1}_{\{S_k=x\}}
\end{equation}
denotes the number of visits to the point $x\in \mathbb{Z}$ up to time $n \in \mathbb{N}$ and $S_k=X_1+\ldots X_k$ is the random walk performed. 
\\

The specific context in which our model is investigated is an extension of the one presented in Section 1.2 of~\cite{DG} and goes as follows. Let $(S_n)_{n\geq 0}$ be a random walk on $\mathbb{Z}$ such that
\begin{equation}
\frac{1}{a_n}S_n \Rightarrow Z_\beta, 
\end{equation}
where $Z_\beta$ has symmetric $\beta$-stable  distribution $1<\beta<2$. In particular, we assume that the random walk is recurrent. In the most general setting $(a_n)_{n\geq 1}$ is regularly varying at infinity with exponent $\beta$. We will assume more, i.e., that $(S_n)$ is in the \emph{normal domain of attraction} of $Z_\beta$ and take $a_n=n^{1/\beta}$. Let $\xi = (\xi_x)_{x\in \mathbb{Z}}$ be a family of i.i.d. random variable such that
\begin{equation}
\frac{1}{n^{1/\alpha}}\sum_{x=0}^n \xi_x \Rightarrow Z_\alpha,
\end{equation}
where $Z_\alpha$ is a symmetric $\alpha$-stable random variable with $\alpha \in (0,2)$. What is different from the model considered in~\cite{DG} is that we introduce more randomness to the model with an i.i.d. sequence $(Y_n)_{n \geq 1}$ such that
\begin{equation}
\frac{1}{n^\gamma}\sum_{j=1}^n Y_j \Rightarrow Z_\gamma,
\end{equation}
where $Z_\gamma$ has a symmetric $\gamma$-stable distribution with $\alpha<\gamma \leq 2$. In the original formulation of~\cite{DG} all the $Y_n$'s are equal to one. For technical reasons we will also assume that
\begin{equation}\label{COND}
\sup_{k \in \mathbb{N}}\mathbb{E}\Bigg|\frac{Y_1+\ldots +Y_k}{k^{1/\gamma}}\Bigg|^{\alpha \kappa} < \infty,
\end{equation}
for some $\kappa>1$. The above condition can be viewed as a restriction on the distributution of $Y_1$. A sufficient condition for~\eqref{COND} to hold is given in the lemma below. We denote the characteristic function of $Y_1$ by $\phi$.

\begin{lem}\label{Alem}
If $\alpha>1$, then~\eqref{COND} is satisfied as long as 
\begin{equation}
\int_r^\infty \frac{|\phi'(\theta)|}{\theta^{\alpha \kappa}}d\theta<\infty
\end{equation}
for some $r>0$ and there is a finite constant $K$ such that $|\phi'(\theta)|\leq K|\theta|^{\gamma-1}$ for $\theta$ in some neighbourhood of zero.
\end{lem}
The proof of Lemma~\ref{Alem} is given in the Appendix.\\

The base for our study is the behaviour of the process
\begin{equation}\label{theZ}
\widetilde{Z}(t):=\sum_{x \in \mathbb{Z}}\Big(\sum_{k=1}^{N_{[t]}(x)}Y_k\Big)\xi_x, \quad t\geq 0.
\end{equation}
We also define the rescaled version of~\eqref{theZ} by
\begin{equation}\label{theD}
D_n(t):=r_n^{-1}\widetilde{Z}(nt),\quad  n \geq 1, i\geq 1, t \geq 0,
\end{equation}
with $r_n=n^{1/\gamma+1/(\alpha\beta) - 1/(\gamma\beta)}.$ 

We are interested  in the scaling limit in which we consider the aggregate behaviour of a large number of independent walkers with independent strategies and having independent environments from which they collect the rewards. More precisely, consider an i.i.d. sequence of processes $\big((D_n^{(i)}(t))_{t\geq 0}\big)_{i=1}^\infty$, $n\geq 1$ and define for $t\geq 0$
\begin{equation}\label{theG}
G_n(t):=\frac{1}{c_n^{1/\alpha}}\sum_{i=1}^{c_n} D_n^{(i)}(t), \quad n\geq 1,
\end{equation}
where $c_n$ is any sequence of positive integers converging to $+\infty$.
Now we may state our result concerning the scaling limit of the above process.

\begin{theorem}\label{THM1}
For any $0<\alpha<\gamma\leq 2$ the process $(G_n(t))_{t\geq 0}$ defined by~\eqref{theG} converges (up to a multiplicative constant) as $n \rightarrow \infty$, in the sense of finite-dimensional distributions, to the process given by~\eqref{identi}.
\end{theorem}

\section{Proof of Theorem~\ref{THM1}}

For clarity we divided the proof of Theorem~\ref{THM1} into a number of lemmas. Basically, we prove the convergence of finite-dimensional distributions by showing the convergence of appropriate characteristic functions. First we will state them and then proceed to their proofs. In order to simplify the notation we put
\begin{equation}
\widetilde{N}_n(x):=\sum_{j=1}^{N_n(x)}Y_j, 
\end{equation}
for $n\in \mathbb{N}$ and $x\in \mathbb{Z}$. Since we are going to work a lot with the characteristic function of $\xi_0$ we introduce the following notation. Let
\begin{equation}
\lambda(u)=\mathbb{E}(exp(iu\xi_0)), \quad u \in \mathbb{R}
\end{equation}
and
\begin{equation}
\bar{\lambda}(u)=exp(-|u|^\alpha), \quad u \in \mathbb{R}.
\end{equation}

Assume that $\theta_1,\ldots,\theta_k\in \mathbb{R}$, $t_1,\ldots, t_k \in  [0,\infty)$ for $k\geq 1$. We want to show the convergence of the characteristic function of $\sum_{j=1}^k \theta_j G_n(t_j)$ to the corresponding characteristic function of the process given by~\eqref{G2}.

The first lemma in this section removes the first layer of randomness in our scheme and expresses the characteristic function in question solely in terms of the random walk and the sequence $(Y_k)_{k\geq 1}$.

\begin{lem}\label{C1}
For the setting as in Section~\ref{DR}
\begin{equation}
\mathbb{E}\big(exp\left(i\sum_{j=1}^k \theta_j G_n(t_j)\right)\big)=\Bigg(\mathbb{E}\Big( \prod_{x \in \mathbb{Z}} \lambda \big(c_n^{-1/\alpha}r_n^{-1}\sum_{j=1}^k \theta_j \sum_{m=1}^{N_{[nt_j]}(x)}Y_m\big)\Big)\Bigg)^{c_n}.
\end{equation}
\end{lem}
The second lemma says that in the limit only the asymptotic behaviour of $\lambda$ near zero matters. 
\begin{lem}\label{C2}
\begin{equation}\label{ueq}
\mathbb{E}\Bigg(c_n\Big(\prod_{x\in \mathbb{Z}}\lambda(c_n^{-1}r_n^{-1}\sum_{j=1}^k \theta_j \widetilde{N}_{[nt_j]}(x))- \bar{\lambda}(c_n^{-1}r_n^{-1}\sum_{j=1}^k \theta_j \widetilde{N}_{[nt_j]}(x))\Big)\Bigg)
\end{equation}
converges to $0$ as $n \rightarrow \infty$.
\end{lem}
The third lemma is the backbone of the whole proof.
\begin{lem}\label{C3}
Let 
\begin{equation}
B_n:= \sum_{x\in \mathbb{Z}}\left|r_n^{-1}\sum_{j=1}^k \theta_j\sum_{m=1}^{N_{[nt_j]}(x)}Y_m\right|^{\alpha}, \quad n \geq 1.
\end{equation}
Then,
\begin{equation}\label{C3_B}
\lim_{n\rightarrow \infty}\mathbb{E}(B_n)=c(\alpha)\mathbb{E}\left(\int_\mathbb{R}\left|\sum_{j=1}^k \theta_j Y(L_{t_j}(x))\right|^\alpha dx\right),
\end{equation}
and 
\begin{equation}\label{bf}
\mathbb{E}(exp(-c_n^{-1}B_n))= 1-c_n^{-1}c(\alpha)\mathbb{E}(B)+o(c_n^{-1}).
\end{equation}
Here $B = \int_\mathbb{R}|\sum_{j=1}^k \theta_j Y(L_{t_j}(x)|^\alpha dx$ and $c(\alpha)$ is a constant depending only on $\alpha$.
\end{lem}

It is evident that given the lemmas above, Theorem~\ref{THM1} follows immediately (see the proof of Theorem 1.2 in~\cite{DG}). First, however, we will show that the random variables $B_n$, $n \in \mathbb{N}$ introduced in the formulation of Lemma~\ref{C3} are uniformly integrable. We do this by showing that $\mathbb{E}|B_n|^\kappa$ is bounded uniformly in $n \in \mathbb{N}$ for some $\kappa>1$. 

\begin{lem}\label{uilem}
Assume that
\begin{equation}\label{mbound}
\sup_{k \in \mathbb{N}}\mathbb{E}\Bigg|\frac{Y_1+\ldots +Y_k}{k^{1/\gamma}}\Bigg|^{\alpha \kappa} < \infty.
\end{equation}
Then there is a constant $C$, independent of $n \in \mathbb{N}$, such that for all $\kappa>1$ sufficiently close  to $1$ we have
\begin{equation}
\mathbb{E}(B_n^\kappa) \leq C.
\end{equation}
\end{lem}

\begin{proof}[Proof of Lemma~\ref{uilem}]
It is enough to prove the lemma with $k=1$ and $\theta_1=1$. Fix $n\in \mathbb{N}$ and $t\geq 0$. Let $x_1,\ldots, x_{s_n}$ be the points in the range of the random walk up to time $[nt]$ taken in the increasing order with respect to $N_{[nt]}(x_i)$. We can write
\begin{equation}
B_n = \frac{1}{r_n^\alpha}\left(\big|Y_1 +\ldots + Y_{N_{[nt]}(x_1)}\big|^\alpha + \ldots + \big|Y_1 +\ldots + Y_{N_{[nt]}(x_{s_n})}\big|^\alpha\right).
\end{equation}
Notice that by Jensen inequality, for any $\kappa>1$ we have
\begin{equation}
B_n^{\kappa} \leq r_n^{-\kappa \alpha} R_{[nt]}^{\kappa - 1}\left(\big|Y_1 +\ldots + Y_{N_{[nt]}(x_1)}\big|^{\alpha \kappa} + \ldots + \big|Y_1 +\ldots + Y_{N_{[nt]}(x_{s_n})}\big|^{\alpha \kappa}\right).
\end{equation}
Since the sequence $(Y_n)_{n\in \mathbb{N}}$ and the random walk are independent, by conditioning on the random walk, we get
\begin{eqnarray}
\mathbb{E}(B_n^\kappa)& \leq & r_n^{-\kappa \alpha} \sup_{k \in \mathbb{N}}\mathbb{E}\Bigg|\frac{Y_1+\ldots +Y_k}{k^{1/\gamma}}\Bigg|^{\alpha \kappa} \nonumber\\
&& \: \mathbb{E}\Big(R_{[nt]}^{\kappa -1}N_{[nt]}(x_1)^{\frac{\alpha \kappa}{\gamma}} +\ldots + N_{[nt]}(x_{s_n})^{\frac{\alpha \kappa}{\gamma}}\Big),
\end{eqnarray}
where $R_m=\sum_{x\in \mathbb{Z}}\mathbf{1}_{\{N_m(x)\neq 0\}}$ for $m\in \mathbb{N}$. We now claim that 
\begin{equation}\label{rbound}
r_n^{-\kappa \alpha}\mathbb{E}\left(R_{[nt]}^{\kappa-1}\sum_{k=1}^{R_{[nt]}} N_{[nt]}(x_k)^{\frac{\alpha \kappa}{\gamma}}\right) 
\end{equation}
is bounded uniformly in $n \in \mathbb{N}$ for all $\kappa>1$ sufficiently close to $0$. Using H\"{o}lder inequality with $p=\frac{\gamma}{\alpha \kappa}$ and $q=\frac{\gamma}{\gamma - \alpha \kappa}$ we see that~\eqref{rbound} is no bigger than
\begin{eqnarray}
&&r_n^{-\kappa \alpha}\mathbb{E}\Big(R_{[nt]}^{\kappa -1}\bigg(\sum_{x\in \mathbb{Z}}\mathbf{1}_{\{N_{[nt]}(x)\neq 0 \}}\bigg)^{\frac{\gamma - \alpha \kappa}{\gamma}}[nt]^\frac{\alpha \kappa}{\gamma}\Big)\nonumber\\
&=& r_n^{-\kappa \alpha}\mathbb{E}\Big(R_{[nt]}^{\frac{(\gamma - \alpha)\kappa}{\gamma}}\Big)[nt]^{\frac{\alpha \kappa}{\gamma}}\nonumber\\
&\leq& r_n^{-\kappa \alpha}\bigg(\mathbb{E}(R_{[nt]})\bigg)^{\frac{(\gamma - \alpha)\kappa}{\gamma}}\big(\sum_{x\in \mathbb{Z}}N_{[nt]}(x)\big)^{\frac{\alpha \kappa}{\gamma}}\nonumber\\
& =& r_n^{-\kappa \alpha}\bigg(\mathbb{E}(R_{[nt]})\bigg)^{\frac{(\gamma - \alpha)\kappa}{\gamma}}[nt]^{\frac{\alpha \kappa}{\gamma}}\label{r1},
\end{eqnarray}
where the inequality in~\eqref{r1} follows from H\"{o}lder inequality as long as $\kappa \leq \frac{\gamma}{\gamma - \alpha}$. By Lemma 1 in~\cite{KS}, $\mathbb{E}(R_{[nt]})\leq c_1 [nt]^{1/\beta}$ for some constant $c_1$ depending only on $\beta$. We thus conclude that~\eqref{rbound} can be bounded by
\begin{equation}
c_1 [nt]^{\frac{(\gamma-\alpha)\kappa}{\gamma \beta}}[nt]^{\frac{\alpha \kappa}{\gamma}}n^{-\frac{\kappa \alpha}{\gamma} - \frac{\kappa}{ \beta} + \frac{\kappa \alpha}{\gamma \beta}}, 
\end{equation}
which is bounded uniformly in $n \in \mathbb{N}$.
\end{proof}

The proof of Claim~\ref{C1} is the same as the proof of Lemma 3.4 in~\cite{DG} and, therefore, we skip it and proceed directly to the proof of~\ref{C2}.

\begin{proof}[Proof of Claim~\ref{C2}]
The proof presented here is very similar to the proof of Lemma 3.5 in~\cite{DG}. Recall that, by assumption, 
\begin{equation*}
\lambda(u) = \bar{\lambda}(u)+o(|u|^\alpha),
\end{equation*}
as $u\rightarrow 0$. Let
\begin{equation}
U_n(x):=r_n^{-1}\sum_{j=1}^k\theta_j \widetilde{N}_{[nt_j]}(x),\quad n\in \mathbb{N}, x\in \mathbb{Z}.
\end{equation}
It is easy to see that
\begin{multline}
\left|\prod_{x\in \mathbb{Z}}\lambda(c_n^{-1/\alpha}U_n(x))-\prod_{x\in \mathbb{Z}}\bar{\lambda}(c_n^{-1/\alpha}U_n(x))\right| \\
\leq \sum_{x\in \mathbb{Z}}\left|\lambda(c_n^{-1/\alpha}U_n(x))-\bar{\lambda}(c_n^{-1/\alpha}U_n(x))\right|.
\end{multline}
Therefore~\eqref{ueq} can be bounded by
\begin{equation}\label{ueq1}
c_n \mathbb{E}\left(\sum_{x\in \mathbb{Z}}\left|\lambda(c_n^{-1/\alpha}U_n(x))-\bar{\lambda}(c_n^{-1/\alpha}U_n(x))\right|\right).
\end{equation}
Define $g(v)= |v|^{-\alpha}|\lambda(v)-\bar{\lambda}(v)|$, for $v\neq 0$ and $g(0)=0$. Then $g$ is bounded and continuous. With this notation~\eqref{ueq1} equals
\begin{equation}\label{ueq2}
\mathbb{E}\left(\sum_{x\in \mathbb{Z}}|U_n(x)|^\alpha g(c_n^{-1/\alpha}U_n(x))\right).
\end{equation}
Fix any $\epsilon>0$ and choose $\delta>0$ such that $|z|<\delta$ implies $|g(z)|<\epsilon$. Then, ~\eqref{ueq2} can be bounded by
\begin{equation}
\epsilon \mathbb{E}\big(\sum_{x\in \mathbb{Z}}|U_n(x)|^\alpha\big)+\norm{g}_\infty \mathbb{E}\big(\sum_{x\in \mathbb{Z}}|U_n(x)|^\alpha \mathbf{1}_{\{c_n^{1/\alpha}|U_n(x)|\geq \delta\}}\big),
\end{equation}
which in turn is bounded by
\begin{equation}\label{ueq3}
\epsilon \mathbb{E}\bigg(\sum_{x\in \mathbb{Z}}|U_n(x)|^\alpha\bigg)+\norm{g}_\infty \mathbb{E}\bigg(\sum_{x\in \mathbb{Z}}|U_n(x)|^\alpha \mathbf{1}_{\{\sum_{x\in \mathbb{Z}}|U_n(x)|^\alpha\geq c_n \delta^\alpha\}}\bigg).
\end{equation}
Since, by Lemma~\ref{uilem} the sequence of random variables $(\sum_{x\in \mathbb{Z}}|U_n(x)|^\alpha)_{n\in \mathbb{N}}$ is uniformly integrable, the first sumand in~\eqref{ueq3} is bounded by $\epsilon$ times a constant independent of $n \in \mathbb{N}$ and the second converges to $0$ as $n\rightarrow \infty$. The choice of $\epsilon$ was arbitrary and hence the proof is finished.

\end{proof}

\begin{proof}[Proof of Claim~\ref{C3}]
First we are going to show that~\eqref{C3_B} holds. Without losing generality we may assume that $0\leq t_1\leq \ldots \leq t_k$. For convenience we also put $t_0=0$.

We can rewrite $\mathbb{E}(B_n)$ as
\begin{eqnarray*}
\int_\mathbb{R}&&\mathbb{E}\bigg|(\theta_1+\ldots \theta_k)Z^{(1)}(N_{[nt_1]([a_nx])})\Big(\frac{N_{[nt_1]([a_nx])}}{na_n^{-1}}\Big)^{1/\gamma} \\
&& \:+ (\theta_2+\ldots \theta_k)Z^{(2)}\big(N_{[nt_2]}([a_nx])-N_{[nt_1]}([a_nx])\big)\\
&& \: \times\Big(\frac{N_{[nt_2]}([a_nx])-N_{[nt_1]}([a_nx])}{na_n^{-1}}\Big)^{1/\gamma}\nonumber\\
&& \: + \ldots + \nonumber\\
&& \: + \theta_k Z^{(k)}\big(N_{[nt_k]}([a_nx])-N_{[nt_{k-1}]}([a_nx])\big)\\
&& \:\times\Big(\frac{N_{[nt_2]}([a_nx])-N_{[nt_1]}([a_nx])}{na_n^{-1}}\Big)^{1/\gamma}\bigg|^{\alpha}dx,\nonumber\\\nonumber
\end{eqnarray*}
where $Z^{(1)}(\cdot),\ldots,Z^{(k)}(\cdot)$ are i.i.d. copies of the sequence (we put $Z^{(j)}(0)=0$ for convenience)
\begin{equation}
Z^{(0)}(m)=\frac{1}{m^{1/\gamma}}\left(Y_1+\ldots Y_m\right),\quad m \in \mathbb{N},
\end{equation}
which are independent of the random walk $(S_n)$. By Skorochod representation theorem we may assume that for $j=1,\ldots,k$, $Z^{(j)}(m)$ converges almost surely to $Z^{(j)}$, which has $S\gamma S$ distribution and a the random variables $Z^{(j)}$ are independent. Let 
\begin{eqnarray}\label{qloc1}
C_n &=&\int_\mathbb{R}\bigg|(\theta_1+\ldots \theta_k)Z^{(1)}\times\Big(\frac{N_{[nt_1]}([a_nx])}{na_n^{-1}}\Big)^{1/\gamma} \\
&& \:+ (\theta_2+\ldots \theta_k)Z^{(2)} \times\Big(\frac{N_{[nt_2]}([a_nx])-N_{[nt_1]}([a_nx])}{na_n^{-1}}\Big)^{1/\gamma} \nonumber\\
&& \: + \ldots + \nonumber\\
&& \:+ \theta_kZ^{(k)} \times\Big(\frac{N_{[nt_k]}([a_nx])-N_{[nt_{k-1}]}([a_nx])}{na_n^{-1}}\Big)^{1/\gamma}\bigg|^\alpha \nonumber\\
\end{eqnarray}
We are going to show that $\mathbb{E}(B_n)-\mathbb{E}(C_n)$ converges to $0$ as $n\rightarrow\infty$. For that we will need the inequalities:
\begin{equation}\label{mloc1}
|a^\alpha - b^\alpha|\leq \alpha |a-b|(a^{\alpha-1}+b^{\alpha-1}), \alpha>1, a,b \geq 0,
\end{equation}
and
\begin{equation}
|a^\alpha - b^\alpha|\leq |a-b|^{\alpha}, 0\leq \alpha \leq 1, a,b \geq 0.
\end{equation}
Assume first that $\alpha>1$. Put
\begin{eqnarray*}
A &=&  \bigg|(\theta_1+\ldots \theta_k)Z^{(1)}(N_{[nt_1]([a_nx])})\Big(\frac{N_{[nt_1]([a_nx])}}{na_n^{-1}}\Big)^{1/\gamma} \\
&& \:+ (\theta_2+\ldots \theta_k)Z^{(2)}\big(N_{[nt_2]}([a_nx])-N_{[nt_1]}([a_nx])\big)\\
&& \: \times\Big(\frac{N_{[nt_2]}([a_nx])-N_{[nt_1]}([a_nx])}{na_n^{-1}}\Big)^{1/\gamma}\nonumber\\
&& \: + \ldots + \nonumber\\
&& \: + \theta_k Z^{(k)}\big(N_{[nt_k]}([a_nx])-N_{[nt_{k-1}]}([a_nx])\big)\\
&& \:\times\Big(\frac{N_{[nt_2]}([a_nx])-N_{[nt_1]}([a_nx])}{na_n^{-1}}\Big)^{1/\gamma}\bigg|,
\end{eqnarray*}
and
\begin{eqnarray*}
B &=& \bigg|(\theta_1+\ldots \theta_k)Z^{(1)}\times\Big(\frac{N_{[nt_1]}([a_nx])}{na_n^{-1}}\Big)^{1/\gamma} \\
&& \:+ (\theta_2+\ldots \theta_k)Z^{(2)} \times\Big(\frac{N_{[nt_2]}([a_nx])-N_{[nt_1]}([a_nx])}{na_n^{-1}}\Big)^{1/\gamma} \nonumber\\
&& \: + \ldots + \nonumber\\
&& \:+ \theta_kZ^{(k)} \times\Big(\frac{N_{[nt_k]}([a_nx])-N_{[nt_{k-1}]}([a_nx])}{na_n^{-1}}\Big)^{1/\gamma}\bigg|.
\end{eqnarray*}
Then by~\eqref{mloc1} and H\"{o}lder inquality
\begin{eqnarray*}
\mathbb{E}|a^\alpha - b^\alpha|&\leq& \alpha \mathbb{E}\big(|A-B|(A^{\alpha-1}+B^{\alpha-1})\big)\\
&\leq& \alpha \big(\mathbb{E}|A-B|^\alpha\big)^{1/\alpha}\Big(\big(\mathbb{E}A^{\alpha}\big)^{(\alpha-1)/\alpha}+\big(\mathbb{E}B^{\alpha}\big)^{(\alpha-1)/\alpha}\Big).
\end{eqnarray*}
By triangle inequality 
\begin{eqnarray}
|A-B|&\leq& \sum_{j=1}^k \big|\theta_j +\ldots + \theta_k\big| \bigg| \Big(Z^{(j)}\big(N_{[nt_j]}([a_nx])-N_{[nt_{j-1}]}([a_nx])\big)-Z^{(j)}\Big)\nonumber\\
&& \: \times \bigg(\frac{N_{[nt_j]}([a_nx])-N_{[nt_{j-1}]}([a_nx])}{na_n^{-1}}\bigg)^{1/\gamma}\bigg|.
\end{eqnarray}
Notice that by~\eqref{COND} the sequence of random variables 
\begin{equation}
\Big(\Big|\frac{Y_1+\ldots Y_n}{n^{1/\gamma}}\Big|^\alpha\Big)_{n\geq 1}
\end{equation}
is uniformly integrable and hence, by conditioning on the random walk and using triangle inequality once again (now for the $\alpha$-norm of a random variable), we conclude that
\begin{eqnarray}
\big(\mathbb{E}|A-B|^\alpha\big)^{1/\alpha} &\leq&  \sum_{j=1}^k \big|\theta_j +\ldots + \theta_k\big|\\
&& \: \times \bigg(\mathbb{E}\Big|f\Big(N_{[nt_j]}([a_nx])-N_{[nt_{j-1}]}([a_nx])\Big)\\
&& \: \times \Big(\frac{N_{[nt_j]}([a_nx])-N_{[nt_{j-1}]}([a_nx])}{na_n^{-1}}\Big)^{1/\gamma}\bigg|^\alpha\bigg)^{1/\alpha}
\end{eqnarray}
where $f:\mathbb{N}\cup\{0\}\rightarrow \mathbb{R}_+$ is a bounded function such that $\lim_{m\rightarrow \infty}f(m)=0$. Using~\eqref{COND} again one can easily notice that both $\mathbb{E}A^{\alpha}$ and $\mathbb{E}B^{\alpha}$ can be bounded by 
\begin{equation}
c_1\mathbb{E}\Big(\frac{N_{[nt_k]([a_nx])}}{na_n^{-1}}\Big)^{\alpha/\gamma}
\end{equation}
for some finite constant $c_1$ independent of $n$. Thus, to show that $|\mathbb{E}(B_n)-\mathbb{E}(C_n)|$ goes to zero as $n\rightarrow\infty$ it remains to prove that for any $j=1,\ldots,k$ 
\begin{eqnarray}\label{mloc2}
\int_\mathbb{R} &&\bigg(\mathbb{E}\Big(f\big(N_{[nt_j]}([a_nx])-N_{[nt_{j-1}]}([a_nx])\big)^\alpha\nonumber\\
&& \: \times \Big(\frac{N_{[nt_j]}([a_nx])-N_{[nt_{j-1}]}([a_nx])}{na_n^{-1}}\Big)^{\alpha/\gamma}\Big)\bigg)^{1/\alpha} \nonumber\\
&& \: \times \bigg(\mathbb{E}\Big(\frac{N_{[nt_k]([a_nx])}}{na_n^{-1}}\Big)^{\alpha/\gamma}\bigg)^{(\alpha-1)/\alpha}dx
\end{eqnarray}
converges to $0$ as $n\rightarrow \infty$. The integrand in~\eqref{mloc2} is bounded by the function

\begin{equation}
x \mapsto c_2\mathbb{E}\Big(\frac{N_{[nt_k]([a_nx])}}{na_n^{-1}}\Big)^{\alpha/\gamma},
\end{equation}
for some constant $c_2$ independent of $n$. It follows from the proof of Lemma 6 in~\cite{KS} that for any $K>0$ and $t>0$
\begin{equation}
\int_{|x|>K}\Big(\frac{N_{[nt]([a_nx])}}{na_n^{-1}}\Big)^{\alpha/\gamma}dx
\end{equation}
converges in distribution to
\begin{equation}\label{qloc4}
\int_{|x|>K}L_t(x)^{\alpha/\gamma}dx
\end{equation}
were $(L_t(x))_{t\geq 0, \in \mathbb{R}}$ is a jointly continuous version of local time of symmetric a $\beta$-stable L\'{e}vy process. By Lemma 3.3 in~\cite{DG} the convergence holds also in $L^1(\Omega)$. Since the expected value of~\eqref{qloc4} converges to $0$ as $K\rightarrow \infty$ (see Lemma 2.1 in~\cite{DG}), we see that by choosing $K$ large enough,
\begin{equation}
\int_{|x|>K}\mathbb{E}\Big(\frac{N_{[nt_k]([a_nx])}}{na_n^{-1}}\Big)^{\alpha/\gamma}dx
\end{equation}
can be made arbitrarily small for all $n$ large enough. Thus it remains to show that for any $K>0$
\begin{eqnarray}\label{mloc3}
\int_{|x|\leq K} &&\bigg(\mathbb{E}\Big(f\big(N_{[nt_j]}([a_nx])-N_{[nt_{j-1}]}([a_nx])\big)^\alpha\nonumber\\
&& \: \times \Big(\frac{N_{[nt_j]}([a_nx])-N_{[nt_{j-1}]}([a_nx])}{na_n^{-1}}\Big)^{\alpha/\gamma}\Big)\bigg)^{1/\alpha} \nonumber\\
&& \: \times \bigg(\mathbb{E}\Big(\frac{N_{[nt_k]([a_nx])}}{na_n^{-1}}\Big)^{\alpha/\gamma}\bigg)^{(\alpha-1)/\alpha}dx
\end{eqnarray}
converges to zero as $n\rightarrow \infty$. This is relatively easy and we will only sketch the idea. Fix any $r>0$ and $j=1,\ldots,k$. The integral in~\eqref{mloc3} can be written as a sum of two integrals $I_1$, $I_2$ depending on whether
\begin{equation}
\frac{N_{[nt_j]([a_nx])}-N_{[nt_{j-1}]([a_nx])}}{na_n^{-1}}
\end{equation}
is greater than $r$ or not. In the first case, taking $n$ sufficiently large, the integrand can be bounded by an arbitrarily small constant (in this case $N_{[nt_j]([a_nx])}-N_{[nt_{j-1}]([a_nx])}$ must be large since $na_n^{-1} \rightarrow \infty$). In the second case  we simply bound the integrand by  
\begin{equation}
r^{1/\gamma}\Big(\frac{N_{[nt_k]([a_nx])}}{na_n^{-1}}\Big)^{(\alpha-1)/\gamma}
\end{equation}
and the corresponding integral (again by Lemma 3.3 in~\cite{DG}) can be bounded from above by a constant independent of $n$ times $c^{1/\gamma}$. Choosing $r$ small in the first place gives us what was needed. The case $0\leq \alpha \leq 1$ is very similar and we skip the proof.
\\

Now, by the stability and independence of $Z^{(1)},\ldots,Z^{(k)}$, $\mathbb{E}(C_n)$ is equal to
\begin{eqnarray}\label{qloc7}
\sum_{x\in \mathbb{Z}}&&r_n^{-\alpha}\Big||\theta_1+\ldots \theta_k|^\gamma N_{[nt_1]}(x)\label{bsum1} \\
&& \:+ |\theta_2+\ldots \theta_k|^\gamma \big(N_{[nt_2]}(x)-N_{[nt_1]}(x)\big)\nonumber\\
&& \: + \ldots + \nonumber\\
&& \: + |\theta_k|^\gamma \big(N_{[nt_k]}(x)-N_{[nt_{k-1}]}(x)\big)\Big|^{\alpha/\gamma}\Bigg) \mathbb{E}\big(|Y_1|^{\alpha}\big) \nonumber
\end{eqnarray}
By Lemmas 3.2 ad 3.3 in~\cite{DG}, ~\eqref{qloc7} converges as $n\rightarrow \infty$, to 
\begin{eqnarray}\label{Xfdd}
\int_{\mathbb{R}}\mathbb{E}\Bigg|\sum_{j=1}^k \big(|\theta_j+\ldots +\theta_k|^\gamma - |\theta_{j+1}+\ldots +\theta_k|^\gamma\big)L_{t_j}(x)\Bigg|^\alpha dx,
\end{eqnarray}
which finishes the proof of~\eqref{C3_B}. Now let us turn to~\eqref{bf}. Define $f_n(x):=c_n(1-exp(-c_n^{-1}(x)))$ for $x\in \mathbb{R}, n \in \mathbb{N}$. Then, ~\eqref{bf} is equivalent to
\begin{equation}\label{bfeq}
\lim_{n\rightarrow \infty}\mathbb{E}f_n(B_n)=\mathbb{E}B.
\end{equation}
We can write, for $\delta>0$
\begin{eqnarray}
\mathbb{E}f_n(B_n)&=& \mathbb{E}\Big(f_n(B_n)\mathbf{1}_{\{|B_n|>c_n^\delta\}}\Big)+\mathbb{E}\Big(f_n(B_n)\mathbf{1}_{\{|B_n|\leq c_n^\delta\}}\Big)\label{bfb}\\
&=& I_1+\mathbb{E}\Bigg(c_n\Big(1-\big(1-B_n/c_n+O((B_n/c_n)^2)\big)\Big)\mathbf{1}_{\{|B_n|\leq c_n^\delta\}}\Bigg),\nonumber
\end{eqnarray}
where (using $|f_n(x)|\leq |x|$ for all $x\in \mathbb{R}$ and $n\in \mathbb{N}$)
\begin{equation}
|I_2(x)| \leq \mathbb{E}\left(|B_n|\mathbf{1}_{\{|B_n|> c_n^\delta\}}\right),
\end{equation}
which converges to $0$ as $n\rightarrow \infty$ by the uniform integrability of $(B_n)_{n \geq 1}$. Using this, and taking $\delta<\frac{1}{2}$ we see that (again by the uniform integrability of $(B_n)_{n \geq 1}$)~\eqref{bfeq} holds. 
\end{proof}

\begin{appendices}

\section{}

\begin{proof}[Proof of Lemma~\ref{Alem}]
Take any $\kappa>1$ such that $\alpha \kappa<\gamma$. In the proof $c_1, c_2,\ldots$ will denote constants independent of $k$ and $\theta$. Since the random variable $Y_1$ is symmetric we may write (using Lemma 1.3 in~\cite{MATS})
\begin{equation}
m_k(\alpha \kappa) :=\mathbb{E}\left|\frac{Y_1+\ldots Y_k}{k^{1/\gamma}}\right|^{\alpha \kappa}= c_1 \int_0^\infty \frac{\phi_k'(-\theta)}{\theta^{\alpha \kappa}}d\theta,
\end{equation}
for some constant $c_1$ which depends only on $\alpha$ and $\kappa$. Here $\phi_k$ denotes the characteristic function of $(1/k^{1/\gamma})(Y_1+\ldots Y_k)$. Recall that by $\phi$ we denote the characteristic function of $Y_1$. Since $Y_1$ in the domain of normal attraction of $Z_\gamma$ we conclude (see~\cite{GH} for proofs) that the function
\begin{equation}
\theta \mapsto 1-\phi(\theta)
\end{equation}
is regularly varying at $0$ with exponent $\gamma$ and in particular
\begin{equation}
\lim_{\theta \rightarrow 0}\frac{1-\phi(\theta)}{|\theta|^\gamma}=c_2,
\end{equation}
with $c_2$ being a finite positive constant depending only on $\gamma$. $m_k(\alpha \kappa)$ can be bounded by
\begin{equation}
c_1 \int_0^\infty \frac{|\phi_k'(\theta)|}{\theta^{\alpha \kappa}}d\theta
\end{equation}
which after a change of variables equals
\begin{equation}\label{aloc1}
c_1 \int_0^\infty \frac{|\phi'(\theta)|\big(1-(1-\phi(\theta))\big)^{k-1}}{\theta^{\alpha \kappa}}k^{1-(\alpha \kappa)/\gamma}d\theta.
\end{equation}
Fix $c>0$ such that 
\begin{eqnarray*}
1-\phi(\theta) &\geq& c_3 |\theta|^\gamma,\\
|\phi'(\theta)| &\leq& c_4|\theta|^{\gamma-1},
\end{eqnarray*}
for $|\theta|\leq c$ and some positive constants $c_3$, $c_4$. The integral in~\eqref{aloc1} can be written as $I_1$ + $I_2$, where $I_1$ and $I_2$ are integrals over $(0,c)$ and $(c,\infty)$ respectively. First, notice that 
\begin{equation}
I_1 \leq c_5\int_0^c \frac{|\theta|^{\gamma-1}(1-c_3|\theta|^\gamma)^{k-1}}{\theta^{\alpha \kappa}}k^{1-(\alpha \kappa)/\gamma}d\theta.
\end{equation} 
Since for $z$ close to zero $1-z \sim exp(-z)$, $I_1$ is no bigger than
\begin{equation}
c_6\int_0^c \theta^{\gamma-1-\alpha \kappa}exp(-(k-1)\theta^\gamma)k^{1-(\alpha \kappa)/\gamma}d\theta.
\end{equation}
Changing variables $\theta = (k-1)^{1/\gamma}\theta$ and using Theorem 10.5.6 in~\cite{SPLRD} we conclude that $\limsup_{k\rightarrow\infty}I_1<\infty$. The fact that for any $c>0$, $I_2$ is bounded uniformly in $k \in  \mathbb{N}$ follows directly from the assumptions of Lemma~\ref{Alem}

\end{proof}
\end{appendices}

\bibliographystyle{plain}

\bibliography{publications}

\end{document}